\documentclass[12pt]{amsart}
\usepackage[margin=1in]{geometry}
\usepackage{amsthm,thmtools,amsfonts,amssymb,amsmath}\usepackage{epigraph}
\usepackage{graphicx}
\usepackage[square,authoryear]{natbib}
\let\cite=\citep \usepackage{microtype}
\usepackage{times}
\usepackage{xypic,mathtools}
\usepackage{enumitem}
\usepackage{xtab} \usepackage{caption}
\usepackage{pdflscape} \usepackage{afterpage} \usepackage{float} \usepackage[colorlinks,citecolor=blue,pagebackref]{hyperref}
\usepackage{cleveref}
\newcommand{\be}{\begin{equation}}
\newcommand{\ee}{\end{equation}}
\newcommand{\bes}{\begin{equation*}}
\newcommand{\ees}{\end{equation*}}
\newcommand{\bea}{\begin{eqnarray}}
\newcommand{\eea}{\end{eqnarray}}
\newcommand{\beas}{\begin{eqnarray}}
\newcommand{\eeas}{\end{eqnarray}}
\newcommand{\ben}{\begin{note}}
\newcommand{\een}{\end{note}}
\newcommand{\bexl}{\vskip0.1em\noindent\hrulefill\vskip1em\begin{ExerciseList}}
\newcommand{\eexl}{\end{ExerciseList}\hrulefill}

\newcommand{\bthm}{\begin{theorem}}
\newcommand{\ethm}{\end{theorem}}
\newcommand{\bpro}{\begin{prop}}
\newcommand{\epro}{\end{prop}}
\newcommand{\bcor}{\begin{corollary}}
\newcommand{\ecor}{\end{corollary}}
\newcommand{\bcon}{\begin{conjecture}}
\newcommand{\econ}{\end{conjecture}}
\newcommand{\bp}{\begin{proof}}
\newcommand{\ep}{\end{proof}}
\newcommand{\blem}{\begin{lemma}}
\newcommand{\elem}{\end{lemma}}
\newcommand{\bn}{\begin{note}}
\newcommand{\en}{\end{note}}
\newcommand{\benum}{\begin{enumerate}}
\newcommand{\eenum}{\end{enumerate}}
\newcommand{\bed}{\begin{defn}}
\newcommand{\eed}{\end{defn}}
\newcommand{\brem}{\begin{remark}}
\newcommand{\erem}{\end{remark}}

\newcommand{\btik}{\begin{tikzpicture}\begin{axis}[scale=0.5,axis y line=center, axis x line=middle]}
\newcommand{\etik}{\end{axis}\end{tikzpicture}}

\let\cong=\equiv

\newcommand{\upperRomannumeral}[1]{\uppercase\expandafter{\romannumeral#1}}

\newtheorem{theorem}[equation]{Theorem}     \newtheorem{lemma}[equation]{Lemma}          \newtheorem{corollary}[equation]{Corollary}  \newtheorem{proposition}[equation]{Proposition}

\theoremstyle{definition}

\theoremstyle{definition}
\newtheorem{defn}[equation]{Definition}
\theoremstyle{remark}

\newtheorem{conj}[equation]{Conjecture}

\theoremstyle{definition}
\newtheorem{remark}[equation]{Remark}

\let\isom=\simeq

\newcommand{\Q}{{\mathbb Q}}

\newcommand{\Z}{{\mathbb Z}}

\renewcommand{\int}{\operatorname{int}}

\newcommand{\legendre}[2]{\left(\frac{#1}{#2}\right)}
\begin{document}

\title[]{On the analytic rank of the twin prime elliptic curve $y^2=x(x-2)(x-p)$}\author{Kirti Joshi}\address{Math. department, University of Arizona, 617 N Santa Rita, Tucson
85721-0089, USA.} \email{kirti@arizona.edu}
\date{Preliminary Version: \today}

\thanks{}\subjclass{}\keywords{}

\newcommand{\fixnumberwithin}[1]{
\numberwithin{equation}{#1}
	\numberwithin{theorem}{#1}
	\numberwithin{conj}{#1}
	\numberwithin{lemma}{#1}
	\numberwithin{proposition}{#1}
	\numberwithin{corollary}{#1}
	\numberwithin{defn}{#1}
	\numberwithin{remark}{#1}
	\numberwithin{rem}{#1}
	\numberwithin{question}{#1}
}

\newcommand{\nws}{\fixnumberwithin{section}}
\newcommand{\nwss}{\fixnumberwithin{subsection}}
\newcommand{\nwsss}{\fixnumberwithin{subsubsection}}
 \fixnumberwithin{subsection} 

\begin{abstract}
Let $p\geq 7$ and suppose $(p,p-2)$ are twin prime numbers, in  \cite{hatley2009}, the elliptic curve $E_p:y^2=x(x-2)(x-p)$ was considered in the context of a conjecture by Jason Beers about the Mordell-Weil  ranks of $E_p/\Q$. I show that for $p\cong 3,5\bmod 8$, the analytic rank of $E_p$ is at least one (\Cref{th:main1}) in line with Beers' predictions. This is done by finding a formula (\Cref{th:main2}) for the global root number of $E_p$ for all twin prime pairs. I also show that Beers' conjecture, that for $p\cong 1\bmod 8$ the rank of $E_p$ is two, is false as stated because $E_{73}$ has rank zero.  In the light of \Cref{th:main2}, Beers' conjecture needs to be modified: if $p\cong 1\bmod 8$ then the rank of $E_p$ is zero or two (\Cref{con:beersconj}).
\end{abstract}
\maketitle

\let\cong=\equiv
\numberwithin{table}{subsection}
\makeatletter
\let\c@table\c@equation
\renewcommand\thetable{\thesubsection.\arabic{table}}
\section{Introduction}
\subsection{}
Let $p\geq 7$ be a prime number such that $p-2$ is also a prime number. Thus
$p,p-2$ form a pair of twin primes. It is a well-known conjecture that the number of twin prime numbers is infinite and this expectation stands considerably strengthened by the work of Y. Zhang, J. Maynard and others (see \cite[Theorem 1.2]{kowalski2018}). In \cite{hatley2009}, the following elliptic
curve has been studied
\begin{equation}\label{eq:twc}
E_p:y^2=x(x-2)(x-p).
\end{equation}
If $p,p-2$ are twin primes, I will call  the elliptic curve over $\Q$ given by \eqref{eq:twc} the \textit{twin prime elliptic curve}. It has good reduction outside $\{2,p, p-2\}$. By \Cref{pr:hatley}, the rank of $E_p$ is at most two  (here by the term \emph{Mordell-Weil rank of $E_p$}, one means the rank of the group $E_p(\Q)$ of the $\Q$-rational points of $E_p$). It has been conjectured by Jason Beers (see 
\cite{hatley2009}) that the twin prime curve $E_p$ has Mordell-Weil rank at
most one if $p\cong3,5,7\mod{8}$, and it is shown,  for $p\cong7\mod{8}$, that
the Mordell-Weil rank is zero (\cite[Theorem 1(a)]{hatley2009}) and if $p\cong 5\mod 8$ then the rank is at most one (\cite[Theorem 1(b)]{hatley2009}). The purpose of this brief note is to complement the result of \cite{hatley2009} by proving  the following
theorem.
\begin{theorem}\label{th:main1}
Assume $p\geq 7$ is a prime number  such that $E_p/\Q$ is the twin prime curve \eqref{eq:twc} (this means $p,p-2$ are a twin prime pair). 
\benum[label={\bf(\arabic{*})}]
\item If $p\cong3,5\mod{8}$, then one has
$$
{\rm ord}_{s=1}L(E_p,s)\geq 1.
$$
\item If $p\cong 5\mod 8$, then
the Mordell-Weil rank of $E_p$ is $\leq 1$.
\eenum
\end{theorem}
The hypothesis that $(p,p-2)$ is a pair of twin primes cannot be relaxed in \Cref{th:main1}{\bf(1)} (see \Cref{re:optimal}). The new result (beyond \cite{hatley2009}) is \Cref{th:main1}{\bf(1)}, which is proved by finding a formula for the root number of the twin prime elliptic curve $E_p$. This is the content of \Cref{th:main2}.
The Birch and Swinnerton-Dyer conjecture, together with \Cref{th:main1} and \cite{hatley2009}, imply that for the above sort of twin primes, the Mordell-Weil rank of $E_p$ is exactly equal to one. Numerical examples of curves covered by the theorem are given in \Cref{ss:numerical}. This data is not available on \cite{lmfdb} at the moment. In the remaining case, namely $p\cong 1\bmod 8$, according to \cite{hatley2009}, Jason Beers had conjectured that the rank of $E_p$ is always equal to two. This turns out to be false--counter examples are provided by \Cref{pr:beersconj}. The corrected version of Beers' conjecture is given in \Cref{con:beersconj}--the needed modification is that for $p\cong 1\bmod 8$, the rank of $E_p$ is either zero or two. 
\subsection{Acknowledgments} It is a pleasure to thank Jeffrey Hatley for a stimulating correspondence and for providing \Cref{pr:hatley} (and its proof), the generators of Mordell-Weil group of the curve $E_{4273}$ in \Cref{pr:beersconj}, \Cref{rem:hatley} and a careful reading of the manuscript.
\section{Preliminaries}
\subsection{Invariants of $E_p$}
I write 
\be\label{eq:min} E_p: y^2=x(x-2)(x-p)=x^3-(p+2)x^2+2px.\ee Thus in Tate's notation \cite[Chapter II]{silverman-arithmetic}, 
one has
\be a_1=a_3=a_6=0;a_2=-(p+2);a_4=2p.\ee
Using \cite[Chapter 1]{silverman-arithmetic}, one easily computes the invariants:
\be\label{eq:invariants}
\begin{aligned}
b_2&=&-4(p+2)\\
b_4&=&4p\\
b_6&=&0\\
b_8&=&-4p^2\\
c_4&=&16((p+2)^2-6p)\\
c_6&=&4^3(p+2)((p+2)^2-9p)\\
\Delta=\Delta(E_p)&=&-64p^2(p-2)^2\\
j(E_p)&=&\frac{c_4^3}{\Delta}.
\end{aligned}
\ee
One checks easily from this that $E_p$ and $E_q$ cannot be isomorphic
over $\Q$ if $p,q$ are distinct primes such that $p,p-2$ and $q,q-2$
are all primes.

\begin{proposition}
Two twin prime elliptic curves   $E_p$ and $E_q$ ($p,q\geq 7$) are $\bar{\Q}$-isomorphic  if and only if $p=q$ (and hence $p-2=q-2$).
\end{proposition}
\bp 
By \cite[Chap II, Proposition 1.4(b)]{silverman-arithmetic}, one knows that $E_p$ and $E_q$ are $\bar{\Q}$-isomorphic if and only if their $j$-invariants coincide i.e. $j(E_p)=j(E_q)$. This gives
\be 
\frac{(16((p+2)^2-6p))^3}{-64p^2(p-2)^2}=\frac{(16((q+2)^2-6q))^3}{-64q^2(q-2)^2}
\ee
and hence
\be 
({-64q^2(q-2)^2}){(16((p+2)^2-6p))^3}={(16((q+2)^2-6q))^3}({-64p^2(p-2)^2}).
\ee
Hence
$$0\cong {(16((q+2)^2-6q))^3}({-64p^2(p-2)^2}) \cong 16\cdot 4^3\cdot (-64) \cdot p^2\cdot (p-2)^2\bmod{q}.$$
Since $q$ does not divide $16((q+2)^2-6q)$ and $q\geq 7$ does not divide $64$, one sees that $q|(p(p-2))$ which gives $q=p$ or $q=p-2$. If $p=q$, then the assertion is proved. 

Hence assume $q=p-2$. Then $p,p-2$ and $q=p-2,q-2=p-4$ are twin prime pairs i.e. $p,p-2,p-4$ are all primes forming a prime triple (note that here one uses $p\geq 7$ which gives $p-4\geq 3$). 
So now one is assuming that $p,p-2,p-4$ forms a prime triple and $E_p\isom E_{{p-2}}$ is a $\bar{\Q}$-isomorphism. Then the equality 
$$j(E_p)=j(E_{p-2})$$
gives, on clearing denominator, the equality
\be 
({-64(p-2)^2(p-4)^2}){(16((p+2)^2-6p))^3}={(16((p)^2-6(p-2))^3}({-64p^2(p-2)^2}).
\ee
Noting that $p-4\geq 3$ is a prime number and reading the above equation modulo $(p-4)$ one gets
$$0\cong {(16((p)^2-6(p-2))^3}({-64p^2(p-2)^2}) \bmod{(p-4)}.$$
Since $p-4$ does not divide $p$ or $p-2$;  $p\cong 4\bmod{p-4}$ and $p-2\cong 2\bmod{p-4}$ one sees that
$$0\cong (16(4^2-12))^3(-64)\bmod{(p-4)}$$
which is impossible as $p-4$ is an odd prime. Thus the assumption  $E_p\isom E_{p-2}$  leads to a contradiction. Hence $q=p$ is the only possibility. This completes the proof of the proposition.
\ep

\subsection{The rank bound $0\leq r(E_p)\leq 2$}
The following was pointed out to me by Jeffery Hatley:
\begin{proposition}\label{pr:hatley}
	For any pair of twin primes $(p,p-2)$, the Mordell-Weil rank of the twin-prime elliptic curve $E_p$ \eqref{eq:twc} is at most two.
\end{proposition}
\bp 
Tautologically, $r(E_p)\geq0$. So the non-trivial assertion is the upper bound. This is a consequence of \citep*[Proposition 1.1]{lozano-robledo2008}.  In the notation of  the said proposition, the Mordell-Weil rank $r(E_p)$ of the twin-prime elliptic curve $E_p$ satisfies
$$r(E_p)\leq \nu(A^2-4B)+\nu(B)-1,$$
where  $\nu(n)$ is the number of distinct prime divisors of any integer $n$ and  $A=-(p+2)$ and $B=2p$.
Thus computing $A^2-4B=(-(p+2))^2-4(2p)=p^2+4p+4-8p=(p-2)^2$ and hence $\nu(A^2-4B)=\nu((p-2)^2)=1$ (as $p-2$ is a prime) and $\nu(4B)=\nu(8p)=2$. Thus $r(E_p)\leq 2$. This proves the claim.
\ep

\section{Reduction types at $2,p,p-2$}
\subsection{} I begin with the following lemmas.
\blem\label{le:val-lemma}
Assume $E_p$ is a twin prime elliptic curve with $p\geq 7$. Then the $j$-invariant $j(E_p)$ of $E_p$ is a $2$-adic integer, but $j(E_p)$ is not a $p$-adic integer and nor is it a $(p-2)$-adic integer. 
\elem
\bp Observe from \eqref{eq:invariants} that $j(E_p)$ is a $2$-adic integer as $p,p-2$ are primes and $p\geq 7$ by assumption. Also note from \eqref{eq:invariants} that the valuation of $j(E_p)$ at $p$ is negative as $$c_4\not\cong0\mod p \text{ and } \Delta\cong 0\bmod p.$$ Hence $j(E_p)$ is not a $p$-adic integer.

Next suppose that the $(p-2)$-adic valuation of $j(E_p)$ is non-negative. Then the numerator, $c_4^3=(16((p+2)^2-6p))^3$, of $j(E_p)$ must be divisible by $p-2$ i.e. 
$$c_4=16((p+2)^2-6p)\cong 0\pmod{p-2}.$$
Writing $c_4$ as $$c_4=16((p-2+4)^2-6(p-2+2))$$
and reading this modulo $(p-2)$ gives
$$0\cong c_4\cong  16((p+2)^2-6p)\cong 16((p-2+4)^2-6(p-2+2)) \cong 16(4^2-12)=64\pmod{p-2}.$$
Thus $64\cong 0\pmod{p-2}$ i.e. $(p-2)|2$ and as $p-2$ is a prime number, $p-2=2$. This contradicts my assumption that $p,p-2$ are both prime numbers and $p\geq 7$ (so $p-2\geq 5$). Thus, the $(p-2)$-adic valuation of $j(E_p)$ is negative. This completes the proof of the \Cref{le:val-lemma}.
\ep 

\blem
Let $E_p$ be a twin prime elliptic curve with $p\geq 7$. Then the equation \eqref{eq:min} is a global minimal Weierstrass equation for $E_p$.
\elem
\bp 
From \eqref{eq:invariants}, and as $p,p-2$ are prime numbers, $$\Delta(E_p)=-64p^2(p-2)^2=-2^6\cdot p^2\cdot (p-2)^2$$ is a prime factorization of $\Delta(E_p)$. Hence one sees from this that the valuation of $\Delta(E_p)$ at  each of the primes $2,p,p-2$, is less than $12$. Thus from \cite[Chap. VII, Remark 1.1]{silverman-arithmetic}, one sees that \eqref{eq:min} is a minimal Weierstrass equation at all primes and hence is a global minimal Weierstrass equation for $E_p$. This completes the proof of this lemma.
\ep

\subsection{} I begin with the determination of the reduction types of $E_p$ at the
primes dividing $\Delta(E_p)$. 
\begin{proposition}
Let $E_p$ be the twin prime elliptic curve~\eqref{eq:twc} with $p\geq 7$. Then
\benum[label={\bf(\arabic{*})}]\label{pr:reduction}
\item $E_p$ has good reduction at all primes other than $\{2,p,p-2\}$.
\item $E_p$ has multiplicative reduction at $p$ (resp. $p-2$) which is
split according to $\legendre{-2}{p}$ (resp. $\legendre{2}{p-2}$). 
\item The
curve has type III\ reduction modulo $2$.
\eenum
\end{proposition}
\begin{proof}
Observe that the only primes dividing
$\Delta$ are $2,p,p-2$, hence 
the first assertion {\bf(1)} is proved. 

So it remains to prove {\bf(2,3)}.	One first determines the reduction type at the twin prime pair $p,p-2$ as this is quite easy. As
$p\geq7$, so $c_4\not\cong0\mod p$, so $E_p$ has multiplicative
reduction at $p$. To determine whether it is split or not according
as the tangent cone $y^2+2x^2\mod{p}$ is split or nor. Equivalently
according as $\legendre{-2}{p}$.

Next the reduction type at $p-2$ is also similarly determined. Write the
equation \eqref{eq:twc} as
$$y^2-x(x-p+2-2)(x-2)=0,$$ reading this modulo $p-2$ one sees that the left hand side is
$$y^2-x(x-2)^2\bmod{(p-2)},$$
and expanding this around $(2,0)$ (i.e. expressing in terms of the
generators of the maximal ideal of the singular point) one sees that the left hand side is
$$y^2-(x-2)^3-2(x-2)^2\mod{(p-2)}.$$
Thus the tangent cone is given by $y^2-2(x-2)^2\mod{(p-2)}$ and thus
the reduction at $p-2$ is multiplicative and is split according to
$\legendre{2}{p-2}$. This completes the proof of {\bf(2)}.

For the proof {\bf(3)}, I will use Tate's algorithm (see
\cite[Chapter 4, \S9]{silverman-advanced}) to find the reduction type at the prime $2$. Now one runs through the steps of the algorithm \cite[Page 366]{silverman-advanced}.

As $y^2\cong x(x-2)(x-p)\cong x^2(x-1)\mod 2$, one sees that $(0,0)$ is a singular point and hence the reduction type is not I${}_0$ and as $2|b_2$ one sees that the reduction type is not I${}_n$ for any $n\geq 1$. Since $2|b_2$ and $4|a_6$, the reduction type is not II. Since $4|a_6$, but $8\nmid b_8$, the reduction type is III, 
and the exponent of the conductor $f=v_2(\Delta)=5$. This completes the proof of {\bf(3)}.
\end{proof}

\section{The root number of $E_p$}

\subsection{}
\Cref{th:main1} is proved by computing the root number of $E_p$. This is the content of \Cref{th:main2} proved below.
\begin{theorem}\label{th:main2}
Let $E_p$ be the twin prime curve \eqref{eq:twc}. Then its global root
number $w_{E_p}$ is given by the following formula:
\begin{equation}\label{rootnum}
w_{E_p}=(-1)^{\frac{p-1}{2}+\frac{(p-2)^2-1}{8}}.
\end{equation}
Explicitly
\begin{table}[hp]
	\centering
\begin{tabular}{|c|c|}
		\hline
$p\mod{8}$ & $w_{E_p}$ \\
		\hline
		$1$ & $1$ \\
		\hline
		$3$ & $-1$ \\
		\hline
		$5$ & $-1$ \\
		\hline
		$7$ & $1$\\
		\hline
	\end{tabular}
\end{table}
\end{theorem}

\begin{proof}
For simplicity, I will write $W_E$ instead of $W_{E_p}$ for the global root number of $E=E_p$ and if $q$
is a prime, then $W_q$ instead of $W_{E_p,q}$ for the local root number at $q$. Recall that the
global root number $W_E$ of $E$, is a product of local root numbers:
$$W_E=W_\infty W_2 W_{p} W_{p-2}.$$
Further by \cite[Proposition 1]{rohrlich1993}, one knows that  for all
elliptic curves one has 
\be\label{eq:winf} W_\infty=-1.\ee 
By \cite[Proposition 3(iii)]{rohrlich1993}, the local root number for
primes of semistable reduction is $-1$ for split multiplicative
reduction and $+1$ for non-split multiplicative reduction. From this
one deduces that
\be \label{eq:wp}
\begin{aligned}
W_{p}&=&-\legendre{-2}{p},\\
W_{p-2}&=&-\legendre{2}{p-2}
\end{aligned}
\ee
So it remains to determine $W_2$. I will use \cite[Table 1]{halberstadt1998} for
this purpose. Let $v_2$ be the standard $2$-adic valuation on $\Q$.
For $x\in\Z$,  write $x'=x/2^{v_2(x)}$. Recall from
\eqref{eq:invariants} that
\begin{eqnarray}
\Delta&=&-2^6\cdot p^2\cdot (p-2)^2\\
c_4&=&16((p+2)^2-6p)\\
c_6&=&2^6(p+2)((p+2)^2-9p)
\end{eqnarray}
Hence one sees that $v_2(\Delta)=6$, $v_2(c_4)=4$ and
$v_2(c_6)=6+v_2((p+2)((p+2)^2-9p))\geq 7$. By \Cref{pr:reduction}, the reduction type of $E_p$ at $2$  is III , there is a unique entry in \cite[Table 1]{halberstadt1998} with Kodaira symbol III and with the triple $$(v_2(\Delta),v_2(c_4),v_2(c_6))=(6,4,\geq7).$$
From \cite[Table 1]{halberstadt1998}, to proceed further, one has to check if $c_4'\cong 3\bmod 4$ holds, where
$$c_4'=\frac{c_4}{2^{v_2(c_4)}}=((p+2)^2-6p).$$
Here if $p\cong1\mod{4}$, then $c_4'\cong(3^2-6)\cong 3\mod{4}$
and if $p\cong3\mod{4}$, then $c_4'\cong(5^2-18)\cong3\mod{4}$. Thus
in all cases $c_4'\cong 3\mod{4}$. Now, following \cite[Table 1]{halberstadt1998}, let
$$T(p)=c_4'-\frac{4c_6}{2^7},$$
Then one has \be T(p)=((p+2)^2-6p)-2(p+2)((p+2)^2-9p).\ee According to
\cite[Table 1, Last column]{halberstadt1998}, the root number $W_2$ of $E$ at $2$ is
determined by the residue class of $T(p)\mod{16}$. More precisely,
from \cite{halberstadt1998} one knows that $W_2$ is $+1$ if and only if
$T(p)\cong7,11\mod{16}$. The values of $T(p)\bmod{16}$ for $E_p$ are given in \Cref{tab-2}.
\begin{table}
	\caption{Table of values of $T(p)$}\label{tab-2}
  \centering
  \begin{tabular}{|c|c|c|}
    \hline
$p\mod{16}$ & $T(p)\mod{16}$ & $W_2$ \\ \hline
    1 & 3 & -1 \\  \hline
    3 & 11 & +1 \\  \hline
    5 & 11 & +1 \\  \hline
    7 & 3 & -1 \\  \hline
    9 & 3 & -1 \\  \hline
    11 & 11 & +1 \\  \hline
    13 & 11 & +1 \\  \hline
    15 & 3 & -1 \\
    \hline
  \end{tabular}
\end{table}
From this table one  sees that
$$W_2=\begin{cases}
  -1 & \text{if } p\cong 1,7,9,15\mod{16}\\
  +1 & \text{if } p\cong 3,5,11,13\mod{16}.
\end{cases}
$$
Or, equivalently 
\be\label{eq:w2} 
W_2=-\legendre{2}{p}.
\ee
 Hence one can put all the
formulas \eqref{eq:winf},  \eqref{eq:wp} and \eqref{eq:w2} together and get:
\be
\begin{aligned}
W_E&=&W_\infty W_2 W_p W_{p-2}\\
&=&-1\left(-\legendre{2}{p}\right)\left(-\legendre{-2}{p}\right)\left(-\legendre{2}{p-2}\right)\\
&=&\legendre{-1}{p}\legendre{2}{p-2},\\
&=&(-1)^{{\frac{p-1}{2}+\frac{(p-2)^2-1}{8}}}.
\end{aligned}
\ee

\end{proof}
\bp[Proof of \Cref{th:main1}] 
From \Cref{th:main2} one has
\be 
W_E=\begin{cases}
	+1 & \text{if } p\cong 1,7\bmod 8,\\
	-1 & \text{if } p\cong 3,5\bmod 8.
\end{cases}
\ee
By \citet{conrad-modularity}, the elliptic curve $E_p$ is modular and so the $L$-function $L(E_p,s)$ of $E_p$ satisfies the functional equation \cite[Appendix C, Conjecture 16.3]{silverman-arithmetic}. The validity of the functional equation, together with $W_E=-1$ for $p\cong 3,5\bmod 8$ says that $L(E_p,1)=0$.  Hence the first assertion of \Cref{th:main1} follows from \Cref{th:main2} and the functional equation satisfied by $L(E_p,s)$. This proves \Cref{th:main1}{\bf(1)}. The assertion \Cref{th:main1}{\bf(2)} follows from \cite[Theorem 1]{hatley2009}.
\ep

\bcor\label{cor:vanish-L-val}
For $p\cong3,5\mod{8}$, one  has $L(E_p,1)=0$, and if additionally $L'(E_p,1)\not=0$, then the Mordell-Weil rank of $E_p$ is equal to one and the Tate-Shafarevich group of $E_p$ is finite.
\ecor
\bp 
The first part follows from \Cref{th:main1}. If $p\cong 3,5\bmod 8$ and $L'(E_p,1)\neq 0$, then one knows by modularity of elliptic curves \cite{conrad-modularity} and the work of Kolyvagin and others  (see \cite[Theorem 1.1.2]{perrin-riou1990}) that the Mordell-Weil rank of $E_p$ is one and the Tate-Shafarevich group of $E_p$ is finite.
\ep

\brem\label{re:optimal}
In  \Cref{cor:vanish-L-val} (and in \Cref{th:main1} and \Cref{con:beersconj}), the hypothesis that $(p,p-2)$ is a pair of twin primes cannot be relaxed. For example, the primes $p= 11,
53, 59, 67, 83, 131, 149, 163, 179, 211, 227$ are all congruent to $3\bmod 8$ or $5\bmod8$, but $p-2$ is not a prime and the Mordell-Weil rank of $E_p$ is zero and the analytic rank is also zero.
\erem

\brem
Note that \Cref{th:main2}, the parity conjecture for elliptic curves together with $0\leq rank(E_p)\leq 2$ (\Cref{pr:hatley}) imply that for $p\cong 3,5\bmod 8$, the Mordell-Weil rank of $E_p$ is equal to one. The case $p\cong 3\bmod 8$ is not covered by \cite{hatley2009}.
\erem

\section{Numerical examples}\label{ss:numerical}
\subsection{}
Let $p=109$, with $p-2=107$ so that \eqref{eq:twc} gives the twin prime elliptic curve $$E_{109}:y^2=x^3-3889x-93240$$
with LMFDB label 373216.d3. Moreover, $p\cong 5\bmod 8$. Then from \Cref{th:main1} one sees that the analytic rank is at least one and the Mordell-Weil rank is at most one. The Birch and Swinnerton-Dyer conjecture predicts that the rank of $E_{109}$ is exactly one. This is confirmed by data provided by \cite{lmfdb} as $E_{109}$ has a rational point which generates $E_p(\Q)$:
$$\left(-\frac{81266967324}{2270046025}\,:\, \frac{1100193936926826}{108156342861125}\,:\,1\right),$$
and
$$L'(E_{109},1)=7.3247752038636664325804904593\cdots $$
and hence the analytic rank is also equal to one.
\subsection{The case $p\cong 1\bmod 8$}
According to \cite{hatley2009}, Jason Beers had conjectured that for $p\cong 1\bmod 8$, the rank of $E_p$ is always two. This turns out to be false as the counter example given below (\Cref{pr:beersconj}{(\bf1)}) shows:
\begin{proposition}\label{pr:beersconj}\ 
\benum[label={\bf(\arabic{*})}]
\item Let $p=73$ so $p-2=71$ and $(p,p-2)$ is a twin-prime pair with $p\cong 1\bmod 8$. The rank of $E_{73}$ is zero. 
\item Moreover, there are $165$ twin primes $(p,p-2)$ satisfying $p\cong 1\bmod 8$ and $7\leq p\leq 50000$. Of these $165$, only $14$ twin-prime pairs listed in \Cref{tab-4}  have  $E_p$ with analytic rank equal to two.  The rest of the curves have analytic rank zero. 
\item For the first prime $p=4273$ in \Cref{tab-4}, $E_{4273}(\Q)$ is generated by 
$$P_1= \left(\frac{7921}{4225}\, :\, -\frac{8695656}{274625}\, :\, 1\right)$$
and
$$P_2=\left(\frac{16743024274002657408}{26415925819311200929}\, :\, -\frac{8257916574244505457580022586480}{135768417051805457107690354033}\, :\, 1\right).$$ Generators of $E_p(\Q)$ for other curves in \Cref{tab-4} are given in \Cref{tab-5}.
\eenum
\end{proposition}
\newcommand{\hhline}{\\ \hline}
\begin{table}[H]
	\centering
	\caption{{\small All the twin prime pairs $(p,p-2)$, satisfying $7\leq p\leq 50000$, $p\cong1\bmod 8$ and $E_p$ of analytic rank two.}}\label{tab-4}
	\begin{tabular}{|c|c|}
		\hline 
		p	& p-2 \hhline
		4273& 4271\hhline
		5641 & 5639 \hhline
		15361 & 15359 \hhline
		20233 & 20231 \hhline
		21601 & 21599 \hhline
		26713 & 26711 \hhline
		31849 & 31847 \hhline
		33289 & 33287 \hhline
		34129 & 34127 \hhline
		42073 & 42071 \hhline
		44281 & 44279 \hhline
		46273 & 46271 \hhline
		47353 & 47351 \hhline
		47713 & 47711 \hhline
	\end{tabular} 
\end{table}

\bp 
This is established by direct computations. For $p=73$, the 2-descent calculation terminates according to \cite{sage} (and John Cremona's MWrank), \cite{lmfdb} and gives rank zero for $E_{73}$ and one checks that the analytic rank is also zero. This proves {\bf(1)}. For curves $7\leq p\leq 50000$, with $p\cong 1\bmod 8$, SageMath reports that analytic rank of $E_p$ is zero except for the fourteen primes listed in \Cref{tab-4}, where the analytic rank is reported to be two. This proves {\bf(2)}. The generators for $E_{4273}$ were found and provided by Jeffery Hatley using MAGMA (and I thank him for it). This completes the proof of the proposition.
\ep

\brem\label{rem:tabs} 
The generators for Mordell-Weil groups of $E_p$ for $p$ in \Cref{tab-4} can be found in \Cref{tab-5}
\erem

\subsection{} \Cref{pr:beersconj}, together with the root number calculation of \Cref{th:main2}, suggests the following modification to Beers' conjecture \cite{hatley2009}:
\begin{conj}\label{con:beersconj}
Let $p\geq 7$ and let $(p,p-2)$ be a twin-prime pair. Then the rank of the twin-prime elliptic curve \eqref{eq:twc} is given by: 
$$r(E_p)=\begin{cases}
0 & \text{ if }p\cong 7\bmod 8,\\
1& \text{ if }p\cong3,5\bmod 8,\\
0 \text{ or } 2 & \text{ if } p\cong1\bmod 8.
\end{cases}
$$
\end{conj}
\brem 
To summarize what is now known about the above conjecture: \cite{hatley2009} establishes the $p\cong 7\bmod 8$ case; if $p\cong 5\mod 8$, \cite{hatley2009} establishes that $r(E_p)\leq 1$; this paper shows that for $p\cong 3,5\bmod 8$, the analytic rank of $E_p$ is at least one; and if $p\cong1\bmod 8$, this paper shows that both the rank zero and rank two cases occur. So the above version of Jason Beers' conjecture is optimal.
\erem

\brem\label{rem:hatley} 
As Hatley remarked, the $E_{4273}$ example in \Cref{pr:beersconj} (and the examples in \Cref{tab-5})  show that the bound of \citep*[Proposition 1.1]{lozano-robledo2008} is sharp in general.
\erem

\subsection{The cases $p\cong 3, 5\bmod8$}
\Cref{tab-3} provides calculations for $E_p$  for twin primes pairs $p,p-2$ with $$5\leq p\leq 5000.$$  A rational point, and its N\'eron-Tate height is either found--the cutoff height is $h\approx 52.0$ in the search for rational points, or reported as undetermined. For all the curves in \Cref{tab-3},  the analytic rank is computationally found to be equal to one i.e. $L'(E_p,1)\neq0$ for each curve in this table. Thus, one expects to find rational points of infinite order on each $E_p$ in the table, but such a point has height higher than the search cutoff height. 

Most of these curves do not occur in the modular forms database \cite{lmfdb}. For example, the curves for $p=139$, $p=1051$ are not presently available in the data provided by \cite{lmfdb}. 
\subsection{Data Tables for $p\cong 1\bmod 8$ and $p\cong 3,5\bmod 8$}
\afterpage{
\newgeometry{margin=0.25in}
\thispagestyle{empty}
\begin{landscape}
\renewcommand{\arraystretch}{1.75}
\begin{table}[H]\centering
\caption{Generators of the Mordell-Weil groups of $E_p$ for $p$ in \Cref{tab-4}}\label{tab-5}
{
\begin{tabular}{|c|l|}
\hline
prime & generators  of $E_p(\Q)$ ($p\cong 1\bmod 8$ and rank $>0$ as in \Cref{tab-4})  \\ \hline
4273 & $  \left(\frac{36106850}{7921} : \frac{54273594960}{704969} : 1\right), \left(\frac{446729893464140162}{92593641708481} : \frac{100941275048690115767268960}{890987780364719222879} : 1\right)  $  \\ \hline

5641 & $  \left(\frac{2351772794880000}{2317174716235681} : -\frac{8375846944576437069542400}{111541934319600341959921} : 1\right), \left(\frac{17178428419390665065835308450}{2790953339824619095348369} : \frac{650560138071554920036106033809844274891120}{4662607562897020758309969643506152903} : 1\right)  $  \\ \hline

15361 & $  \left(\frac{6552835200}{6705808321} : -\frac{68039263362125520}{549131937598369} : 1\right), \left(\frac{3969148574341442}{213916425121} : \frac{103738747847782393425120}{98938699699138831} : 1\right)  $  \\ \hline

20233 & $  \left(\frac{27325089957888}{31772804007169} : -\frac{25223584948106436265920}{179094939940957767553} : 1\right), \left(\frac{817570712940962}{40436881} : -\frac{23365256593725419342640}{257138126279} : 1\right)  $  \\ \hline

21601 &   could not compute generators  \\ \hline

26713 & $  \left(\frac{49136225348825129088}{26236229433477407569} : -\frac{10718300616565883622687234953040}{134385412945936386190695889753} : 1\right)$, \\ {} & $\qquad\qquad\left(\frac{461205977755750654413986458039199042}{2171934872378792077802535369121} : \frac{292850646170881113021605299643991403894277106857909280}{3200885930577147197903757347998785598327714831} : 1\right)  $  \\ \hline

31849 &   could not compute generators  \\ \hline

33289 & $  \left(\frac{62238356378641}{38723799894025} : -\frac{34931057732893290204264}{240972204551534001125} : 1\right)$,\\  {} & $\qquad\qquad\left(\frac{77105180562517113790695092263634374897628450}{455258956186831257132024729877969712209} : -\frac{606878676240360189445421336038333402299797374809140279874668469360}{9713768463664862824526305208795408538530992549157378309927} : 1\right)  $  \\ \hline

34129 &   could not compute generators  \\ \hline

42073 & $  \left(\frac{3536235650}{83521} : \frac{16682476302720}{24137569} : 1\right), \left(\frac{172565144239024331537381481493414776592322}{2997159623851992711466356567852526561} : -\frac{37197309808797429625060704186412263015893482158535455111747840}{5188774655997151312834346439006475843783590976518288559} : 1\right)  $  \\ \hline

44281 &   could not compute generators  \\ \hline

46273 & $  \left(\frac{666830722718144470280335733255739660002}{14200645228471925442306974595870001} : -\frac{2079384337251599251870484313731159974457423821787375864400}{1692240623948830874367400824970819818283002274273751} : 1\right), \left(\frac{246588817}{2304} : \frac{2917399049785}{110592} : 1\right)  $  \\ \hline

47353 & $  \left(\frac{1154568377047218159601}{578893905665666647225} : -\frac{319200603158504088882954992179776}{13928311349129246419731002942125} : 1\right),$ \\ 
{} &  $\qquad \qquad\left(\frac{90118674058001197316956058216861147181858862082}{820059620058375165391374191718731897053441} : -\frac{20408538739309302420516937765716663010075359247857808069709341493321760}{742622565187836622208493318201647166913783314695070519207727489} : 1\right)  $  \\ \hline

47713 &   could not compute generators \\ \hline
\end{tabular}
} \end{table}\end{landscape}
\clearpage
\restoregeometry
}
\newpage \clearpage
\newpage
\begin{landscape}
\tablecaption{$E_p$ for $p\cong3,5\mod 8$ and $5\leq p\leq 5000$}\label{tab-3}
\tabletail{ \hline & &   {\bf continued on the next page} & \\ \hline}
{\tiny
\tablehead{ \hline  twin primes $p$, $p-2$ & $E_p$ & generator $P= [x : y: z]$ for $E_p(\Q)$ & height \\ }	
\begin{xtabular}{|c|c|c|c|}	
\hline
5,3 & $y^2 = x^{3} - 7 x^{2} + 10 x$ & $\left[1: 2: 1\right]$ & 1.03571952245041  \\  \hline 
13,11 & $y^2 = x^{3} - 15 x^{2} + 26 x$ & $\left[650: 16380: 1\right]$ & 4.84635169014375 \\  \hline 
19,17 & $y^2 = x^{3} - 21 x^{2} + 38 x$ & $\left[7650: 17040: 4913\right]$ & 7.09575579081593  \\  \hline 
43,41 & $y^2 = x^{3} - 45 x^{2} + 86 x$ &  $\left[270738: -3625440: 704969\right]$ & 10.8533654519883  \\  \hline 
61,59 & $y^2 = x^{3} - 63 x^{2} + 122 x$ &  $\left[15055405: -44046714: 8615125\right]$ & 13.0334893144444  \\  \hline 
109,107 & $y^2 = x^{3} - 111 x^{2} + 218 x$ & $\left[129820027709645: -1100193936926826: 108156342861125\right]$ & 24.2297718198169  \\  \hline 
139, 137 & $y^2 = x^{3} - 141 x^{2} + 278x$ & [80272604479561650 : -26374799216782336320 : 31212821210965308377] &  32.3921347790875 \\ \hline
181, 179 & $y^2 = x^{3} - 183 x^{2} + 362x$ &  [5605412225034473801 : -53320128951722601150 : 4214169403091442521] & 31.5321218704549 \\ \hline
229, 227 & $y^2 = x^{3} - 231 x^{2} + 458x$ & [90421874515312609088775099425 : -1740570359833703596226117931438 : 118610313178795185457850890625] & 47.692193476146 \\ \hline
283, 281 & $y^2 = x^{3} - 285 x^{2} + 566x$ & undetermined & undetermined \\ \hline
349, 347 & $y^2 = x^{3} - 351 x^{2} + 698x$ & undetermined & undetermined \\ \hline
421, 419 & $y^2 = x^{3} - 423 x^{2} + 842x$ & [82939817244747885380109425 : -7707070095401441707527588486 : 892214815439434640988640625] & 44.738278969034\\ \hline
523, 521 & $y^2 = x^{3} - 525 x^{2} + 1046x$ & [149788982907 : 7424182066400 : 50243409] & 16.6915703108563 \\ \hline
571, 569 & $y^2 = x^{3} - 573 x^{2} + 1142x$ & [1133539873650 : -2381701511280 : 571167214073] & 21.2192523066479 \\ \hline
619, 617 & $y^2 = x^{3} - 621 x^{2} + 1238x$ &  undetermined &  undetermined \\ \hline
643, 641 & $y^2 = x^{3} - 645 x^{2} + 1286x$ & undetermined &  undetermined \\ \hline
661, 659 & $y^2 = x^{3} - 663 x^{2} + 1322x$ &   [24025 : -338334 : 15625] &  10.0300376167424  \\ \hline
811, 809 & $y^2 = x^{3} - 813 x^{2} + 1622 x $ &    undetermined & undetermined \\  \hline
829, 827 & $y^2 = x^{3} - 831 x^{2} + 1658 x $ &   undetermined & undetermined \\  \hline
859, 857 & $y^2 = x^{3} - 861 x^{2} + 1718 x $ & [604731073635 : 28106524887904 : 200201625]    & 17.3782221438833 \\ \hline
883, 881 & $y^2 = x^{3} - 885 x^{2} + 1766 x $ & [19400050779799110498 : -28308142604128838349360 : 23399666066728853924849]    & 37.7296679336938 \\ \hline
1021, 1019 & $y^2 = x^{3} - 1023 x^{2} + 2042 x $ &   undetermined & undetermined \\  \hline
1051, 1049 & $y^2 = x^{3} - 1053 x^{2} + 2102 x $ & [408335188104822034951650 : -6468575775231225538842960 : 252991810494429550428977]    & 39.403076255363 \\ \hline
1093, 1091 & $y^2 = x^{3} - 1095 x^{2} + 2186 x $ &    undetermined & undetermined \\  \hline
1291, 1289 & $y^2 = x^{3} - 1293 x^{2} + 2582 x $ & [290003837202 : -12176696012880 : 343147021001]  & 21.2888752855122\\ \hline
1429, 1427 & $y^2 = x^{3} - 1431 x^{2} + 2858 x $ &    undetermined & undetermined \\  \hline
1453, 1451 & $y^2 = x^{3} - 1455 x^{2} + 2906 x $ &    undetermined & undetermined \\  \hline
1621, 1619 & $y^2 = x^{3} - 1623 x^{2} + 3242 x $ & [98441 : -2506350 : 68921]&      11.4686762044304 \\ \hline
1669, 1667 & $y^2 = x^{3} - 1671 x^{2} + 3338 x $ &    undetermined & undetermined \\  \hline
1699, 1697 & $y^2 = x^{3} - 1701 x^{2} + 3398 x $ &    undetermined & undetermined \\  \hline
1723, 1721 & $y^2 = x^{3} - 1725 x^{2} + 3446 x $ &    undetermined & undetermined \\  \hline
1933, 1931 & $y^2 = x^{3} - 1935 x^{2} + 3866 x $ &    undetermined & undetermined \\  \hline
2029, 2027 & $y^2 = x^{3} - 2031 x^{2} + 4058 x $ &  [1963850683525925 : -181521919007512902 : 5117346936828125] &    28.2684031877315 \\ \hline
2083, 2081 & $y^2 = x^{3} - 2085 x^{2} + 4166 x $ &    undetermined & undetermined \\  \hline
2269, 2267 & $y^2 = x^{3} - 2271 x^{2} + 4538 x $ &    undetermined & undetermined \\  \hline
2341, 2339 & $y^2 = x^{3} - 2343 x^{2} + 4682 x $ &    undetermined & undetermined \\  \hline
2659, 2657 & $y^2 = x^{3} - 2661 x^{2} + 5318 x $ &    undetermined & undetermined \\  \hline
2731, 2729 & $y^2 = x^{3} - 2733 x^{2} + 5462 x $ & [217537433921542482 : -667017137351656080 : 109133803517188841]&    30.123101134861  \\ \hline
2971, 2969 & $y^2 = x^{3} - 2973 x^{2} + 5942 x  $ &  [1475847455709 : 29667458726360 : 437245479]&     17.3896284136565 \\ \hline
3253, 3251 & $y^2 = x^{3} - 3255 x^{2} + 6506 x  $ &    undetermined & undetermined \\  \hline
3259, 3257 & $y^2 = x^{3} - 3261 x^{2} + 6518 x  $ &    undetermined & undetermined \\  \hline
3301, 3299 & $y^2 = x^{3} - 3303 x^{2} + 6602 x  $ &    undetermined & undetermined \\  \hline
3373, 3371 & $y^2 = x^{3} - 3375 x^{2} + 6746 x  $ &    undetermined & undetermined \\  \hline
3469, 3467 & $y^2 = x^{3} - 3471 x^{2} + 6938 x  $ &  [5951119699105739504268125 : -572579250574363992530776422 : 10917131904799712158203125] &    42.8572228480239  \\ \hline
3541, 3539 & $y^2 = x^{3} - 3543 x^{2} + 7082 x  $ &    undetermined & undetermined \\  \hline
3853, 3851 & $y^2 = x^{3} - 3855 x^{2} + 7706 x  $ &  [81652316512036937346595788965 : -37818008303815799471765066208246 : 2313848834018649024325885872125] &    51.0858494094309 \\ \hline
3931, 3929 & $y^2 = x^{3} - 3933 x^{2} + 7862 x  $ &    undetermined & undetermined \\  \hline
4003, 4001 & $y^2 = x^{3} - 4005 x^{2} + 8006 x  $ &    undetermined & undetermined \\  \hline
4021, 4019 & $y^2 = x^{3} - 4023 x^{2} + 8042 x  $ &    undetermined & undetermined \\  \hline
4051, 4049 & $y^2 = x^{3} - 4053 x^{2} + 8102 x  $ &  [77368624902377490652439202 : -1656524070676973243560597200 : 43063878046694171142259601] &    43.502955577273 \\ \hline
4093, 4091 & $y^2 = x^{3} - 4095 x^{2} + 8186 x  $ &    undetermined & undetermined \\  \hline
4219, 4217 & $y^2 = x^{3} - 4221 x^{2} + 8438 x  $ &    undetermined & undetermined \\  \hline
4243, 4241 & $y^2 = x^{3} - 4245 x^{2} + 8486 x  $ &    undetermined & undetermined \\  \hline
4261, 4259 & $y^2 = x^{3} - 4263 x^{2} + 8522 x  $ &    undetermined & undetermined \\  \hline
4339, 4337 & $y^2 = x^{3} - 4341 x^{2} + 8678 x  $ &    undetermined & undetermined \\  \hline
4483, 4481 & $y^2 = x^{3} - 4485 x^{2} + 8966 x  $ &    undetermined & undetermined \\  \hline
4549, 4547 & $y^2 = x^{3} - 4551 x^{2} + 9098 x  $ &    undetermined & undetermined \\  \hline
4651, 4649 & $y^2 = x^{3} - 4653 x^{2} + 9302 x  $ &    undetermined & undetermined \\  \hline
4723, 4721 & $y^2 = x^{3} - 4725 x^{2} + 9446 x  $ &  [893889649690983582373514983465 : -1859166957709761119211990271878 : 447348711300696403727285121625]&    50.0988273176324 \\ \hline
4933, 4931 & $y^2 = x^{3} - 4935 x^{2} + 9866 x  $ &    undetermined & undetermined \\  \hline
\end{xtabular}
}
\end{landscape}

\bibliographystyle{plainnat}

\end{document}